\documentclass[12pt]{article}
\usepackage{amsmath,amsfonts,amsthm,amscd,upref,amstext}
\usepackage[dvips]{graphicx}
\usepackage{subfigure}
\usepackage{pb-diagram,pb-xy}
\usepackage[all]{xy}

\newtheorem{prop}{Proposition}[section]
\newtheorem{thm}[prop]{Theorem}
\newtheorem{cor}[prop]{Corollary}

\newtheorem{defn}[prop]{Definition}
\newtheorem{rem}[prop]{Remark}
\newtheorem{lem}[prop]{Lemma}

\newcommand{\N}{\mathbb{N}}

\numberwithin{equation}{section}

\begin{document}

\title{Mixed multiplicities of arbitrary  modules}
\author{R. Callejas-Bedregal$^{1,}\,$\thanks{Partially supported by CAPES-Brazil Grant Procad-190/2007, CNPq-Brazil Grant 620108/2008-8 and by
FAPESP-Brazil Grant 2010/03525-9. 2000 Mathematics Subject
Classification: 13H15(primary). {\it Key words}: Mixed
multiplicities, (FC)-sequences, Buchsbaum-Rim multiplicity.
}\,\,\,\,and\,\,\,V. H. Jorge P\'erez$^{2}$
\thanks{Work partially supported by CNPq-Brazil - Grant
309033/2009-8, Procad-190/2007,  FAPESP Grant 09/53664-8.}}

\date{}
\maketitle

\noindent $^1$ Universidade Federal da Para\'\i ba-DM, 58.051-900,
Jo\~ao Pessoa, PB, Brazil ({\it e-mail: roberto@mat.ufpb.br}).
\vspace{0.3cm}

\noindent $^2$ Universidade de S{\~a}o Paulo -
ICMC, Caixa Postal 668, 13560-970, S{\~a}o Carlos-SP, Brazil ({\it
e-mail: vhjperez@icmc.usp.br}).

\vspace{0.3cm}
\begin{abstract}

Let $(R, \mathfrak m)$ be a Noetherian local ring. In this work we
extend the notion of mixed multiplicities of modules, given in
\cite{Kleiman-Thorup2} and \cite{Kirby-Rees1} (see also
\cite{Bedregal-Perez}),  to an arbitrary family  $E,E_1,\ldots,
E_q$ of $R$-submodules of $R^p$ with $E$ of finite colength. We
prove that these mixed multiplicities coincide with the
Buchsbaum-Rim multiplicity of some suitable $R$-module. In
particular, we recover the fundamental Rees's mixed multiplicity
theorem  for modules, which was proved  first by Kirby and Rees in
\cite{Kirby-Rees1} and recently also proved by the authors in
\cite{Bedregal-Perez}. Our work is based on, and extend to this
new context, the results on mixed multiplicities of ideals
obtained by Vi\^et in \cite{Viet8} and Manh and Vi\^et in
\cite{Manh-Viet}. We also extend to this new setting some of the
main results of Trung in \cite{Trung} and Trung and Verma in
\cite{Trung-Verma1}. As in \cite{Kleiman-Thorup2},
\cite{Kirby-Rees1} and \cite{Bedregal-Perez}, we actually work in
the more general context of standard graded $R$-algebras.
\end{abstract}

\maketitle
\section{Introduction}

The theory of mixed multiplicities of finitely many
zero-dimensional ideals goes back to the work of Risler and
Teissier in \cite{Teissier} where they use these mixed
multiplicities to study the Whitney equisingulariy of families of
hypersurfaces with isolated singularities. Risler and Teissier
({\it loc. cit.}) also proved that each mixed multiplicity could
be described as the usual Hilbert-Samuel multiplicity of the ideal
generated by an appropriated superficial sequence. This result of
Risler and Teissier was later generalized by Rees in \cite{Rees}
who proved that the mixed multiplicities of a family of
zero-dimensional ideals could be described as the Hilbert-Samuel
multiplicity of the ideal generated by a suitable joint reduction.
This Theorem of Rees is known as Rees's mixed multiplicity theorem
and it is a crucial result in the theory of mixed multiplicities
for zero-dimensional ideals.

In order to extend the above results to the case where not all the
ideals are zero-dimensional, Vi\^et introduced in \cite{Viet8} the
notion of (FC)-sequences and showed that mixed multiplicities of a
family of arbitrary ideals could be described as the
Hilbert-Samuel multiplicity of the ideal generated by a suitable
(FC)-sequence (see also \cite{Manh-Viet}, \cite{Viet11},
\cite{Manh-Viet2}, \cite{Viet10}, \cite{Viet9}). Similar results
were also obtained by Trung in \cite{Trung} for a couple of
ideals, and later generalized for finitely many ideals by Trung
and Verma  in \cite{Trung-Verma1}, by using a stronger notion of
general sequences than that of (FC)-sequence (see \cite{Twoviet}).
Trung and Verma ({\it loc. cit.}) used their results on mixed
multiplicities of ideals to interpret mixed volume of lattice
polytopes as mixed multiplicities of ideals and also to give a
purely algebraic proof of Bernstein's theorem. In general, mixed
multiplicities  have been also mentioned  in the works  of Verma,
Katz, Swanson and other authors (see e.g. \cite{Verma1},
\cite{Verma2}, \cite{Verma3},
\cite{Herrmann-Hyry-Ribbe-Tang},\cite{Katz-Verma},
\cite{Bedregal-Perez1}).

The notion of mixed multiplicities for a family $E_1,\ldots, E_q$
of $R$-submodules of $R^p$ of finite colength, where $R$ is a
local Noetherian ring, have been described in a purely algebraic
form by Kirby and Rees in \cite{Kirby-Rees1} and in an
algebro-geometric form by Kleiman and Thorup in
\cite{Kleiman-Thorup} and \cite{Kleiman-Thorup2}. The main result
of Risler and Teissier in \cite{Teissier} was generalized for
modules in \cite{Bedregal-Perez} and the main result of Rees in
\cite{Rees} was generalized for modules in \cite{Kirby-Rees1} and
\cite{Bedregal-Perez}, where the mixed multiplicities for
$E_1,\ldots, E_q$ are described as the Buchsbaum-Rim multiplicity
of a module generated by a suitable superficial  sequence and
joint reduction of $E_1,\ldots, E_q,$ respectively.

Let $(R, \mathfrak m)$ be a Noetherian local ring. In this work we
extend the notion of mixed multiplicities to an arbitrary family
$E,E_1,\ldots, E_q$ of $R$-submodules of $R^p$ with $E$ of finite
colength in $R^p$ and prove that these mixed multiplicities
coincide with the Buchsbaum-Rim multiplicity of some $R$-modules.
In particular, we recover the fundamental Rees's mixed
multiplicity theorem  for modules, which was proved  first by
Kirby and Rees in \cite{Kirby-Rees1} and recently also proved by
the authors in \cite{Bedregal-Perez}.
 Our work is based on, and
extend to this new context, the results on mixed multiplicities of
ideals obtained by Vi\^et in \cite{Viet8} and Manh and Vi\^et in
\cite{Manh-Viet}. We also extend to this new setting some of the
main results of Trung in \cite{Trung} and Trung and Verma in
\cite{Trung-Verma1}. In fact, we do this in the context of
standard graded algebras.

Fix a graded $R$-algebra  $G=\oplus G_n,$ that, as usual, is
generated as algebra by finitely many elements of degree one and
$M$  a finitely generated graded $G$-module. In Section 2, we
introduce the notion of (FC)-sequences of $R$-submodules of $G_1$
with respect to $M.$ In Section 3, we recall the concept of
Buchsbaum-Rim multiplicities of $R$-submodules of $G_1$  with
respect to $M,$ as defined in \cite{Kleiman-Thorup2} and
\cite{Kirby-Rees1}. In Section 4, we introduce the notion of mixed
multiplicities of a family $J, I_1,\dots, I_q$ of arbitrary
$R$-submodules of $G_1,$with $J$ of finite colength in $G_1$, and
we link these mixed multiplicities and Buchsbaum-Rim
multiplicities via (FC)-sequences of these $R$-submodules of
$G_1.$  The main results of this section are Theorem \ref{prop1},
Theorem \ref{Teo1} and Theorem \ref{Teo4}. In Section 5, we apply
the results on mixed multiplicities of Section 4 to arbitrary
modules (Theorem \ref{mod1} and Theorem \ref{mod2}). In
particular, we get in Theorem \ref{mod3} interesting results
similar to that of Kirby and Rees in \cite{Kirby-Rees1} and the
authors in \cite{Bedregal-Perez}  but in terms of (FC)-sequences.

It is worth noting that, even though D. Q. Viet and N. T. Manh in
\cite{Manh-Viet4} and D. Q. Viet and T. T. H. Thanh in
\cite{Viet-Thanh}, develop the theory of mixed multiplicities of
multigraded  modules over a finitely generated standard
multigraded algebras over an Artinian local ring, their approach
do not apply to obtain our main results because the multigraded
modules we use to define mixed multiplicities, of finitely many
modules, are modules over a finitely generated standard
multigraded algebra over a Noetherian local ring $(R, \mathfrak
m)$, whose support over $R$ is finite, and not over an Artinian
local ring.

\section{FC-sequences}

{\bf Setup (1):} Fix $(R, \mathfrak m)$ an arbitrary Noetherian
local ring; fix a graded $R$-algebra  $G=\oplus G_n,$ that, as
usual, is generated as algebra by finitely many elements of degree
one; fix $J$ a finitely generated $R$-submodule of $G_1$ such that
$\ell(G_1/J)<\infty;$ fix $I_1,\dots, I_q$ with $I_i\subseteq
G_{1}$ finitely generated $R$-submodules; and fix $M=\oplus M_n$ a
finitely generated graded $G$-module generated in degree zero,
that is $M_n=G_nM_0$ for all $n\geq 0.$ We denote by ${\mathcal
I}$ the ideal of $G$ generated by $I_1\cdots I_q.$  Set
$N=0_M:{\mathcal I}^{\infty}$,  $M^*:=M/N,$ $\overline{M}=M/xM$
and $\overline{M}^*= \overline{M}/0_{\overline{M}}:{\mathcal
I}^{\infty}\cong M/xM:{\mathcal I}^{\infty}$

In this work, we will define the notion of mixed multiplicities of
$J,I_1,\dots, I_q$ with respect to $M$ and we will prove that
these mixed multiplicities could be described as Buchsbaum-Rim
multiplicities. In order to relate the above multiplicities we
define the notion of (FC)-sequences for a family of $R$-submodules
of $G_1.$ The results of this work will show that (FC)-sequences
carries important information on mixed multiplicities, as the one
just mentioned.

We use the following multi-index notation through the
remaining part of this work. The norm of a multi-index ${\bf
n}=(n_1,\ldots, n_k)$ is $|{\bf n}|=n_1+\cdots+n_k$ and ${\bf
n}!=n_1!\cdots n_k!.$ If ${\bf n,d}$ are two multi-index then
${\bf n}^{\bf d}=n_1^{d_1}\cdots n_k^{d_k}.$ If ${\bf
I}=(I_1,\ldots,I_k)$ is a $k$-tuple of $R$-submodules of $G_1$
then ${\bf I}^{\bf n}:= I_1^{n_1}\cdots I_k^{n_k}.$  We also use the
following notation, $\delta(i)=(\delta(i,1),\ldots, \delta(i,k)),$
where $\delta(i,j)=1$ if $i=j$ and $0$ otherwise.

\begin{defn}\label{FC}
Let $I_1,\dots,I_q$ be $R$-submodules of $G_1$.  Assume that
$\mathcal I$ is not contained in $\sqrt{\mbox{Ann}M}.$ We say that
an element $x\in G_1$  is an (FC)-element with respect to
$(I_1,...,I_q; M)$ if there exists
 an $R$-submodule $I_i$ of $G_1$ and an integer $r'_i$ such that
\begin{enumerate}
\item [$(FC_1)$] $x\in I_i\setminus {\mathfrak m}I_i$  and
$${\bf I}^{\bf r}M_p\cap xM_{|{\bold r}|+p-1}=x{\bf I}^{{\bf r}-\delta(i)}M_p$$ \noindent for all ${\bold r}\in \N^q$ with
$r_i\geq r'_i.$

\item [$(FC_2)$] $x$ is a filter-regular element with respect to $({\cal I}; M),$ i.e., $0_{M}:x\subseteq 0_{M}:{\mathcal I}^{\infty}.$

\item [$(FC_3)$] $\dim(\mbox{Supp}(M/xM:{\mathcal I}^{\infty}))=\dim(\mbox{Supp}(M^*))-1.$
\end{enumerate}

 \noindent We call  $x\in G_1$ a weak-(FC)-element with respect to $(I_1,...,I_q; M)$ if $x$ satisfies the conditions $(FC_1)$ and $(FC_2).$

 A sequence of elements
$x_1,\ldots, x_k$ of $G_1$,  is said to be an (FC)-sequence with
respect to $(I_1,...,I_q; M)$ if ${\overline x}_{i+1}$ is an
(FC)-element with respect to {\break}
$({\overline{I}}_1,...,{\overline{I}}_q; \overline{M})$ for each
$i=1,...,q-1$, where $\overline{M}=M/(x_1,...,x_i)M$,  ${\overline
x}_{i+1}$ is the initial form of $x_{i+1}$ in
$\overline{G}=G/(x_1,\dots,x_i)$  and ${\overline
I}_i=I_i\overline{G}$, $i=1,\dots, q.$

A sequence of elements $x_1,\ldots, x_k$ of $G_1$,  is said to be
a weak-(FC)-sequence with respect to $(I_1,...,I_q; M)$ if
${\overline x}_{i+1}$ is a weak-(FC)-element with respect to
$({\overline{I}}_1,...,{\overline{I}}_q; \overline{M})$ for each
$i=1,...,q-1$.
\end{defn}

The following result is crucial  for showing the existence of
weak-(FC)-sequences.

\begin{lem}\label{Rees Lemma}(\rm{Generalized Rees'Lemma}) In the setup $(1)$,
let $\Sigma$ be a finite set of prime ideals not containing
$\mathcal I.$ Then for each $i=1,\ldots, q,$ there exists an
element $x\in I_i\setminus {\mathfrak m}I_i,$ $x$ not contained in
any prime ideal in $\Sigma,$ and a positive integer $k_i$ such
that for all $r_i\geq k_i$ and all non-negative integers
$r_1,\dots,r_{i-1},r_{i+1},\dots,r_q$
$$I_1^{r_1}\cdots I_i^{r_i}\cdots I_q^{r_q}M_p\cap xM_{|\bold r|+p-1}=xI_1^{r_1}\cdots I_i^{r_i-1}\cdots I_q^{r_q}M_p.$$
\end{lem}

\begin{proof} Let ${\mathcal I}_i$ be the ideal in $G$ generated
by $I_i, i=1,\dots, q.$ By \cite[Lemma 2.2]{Manh-Viet} for each
$i=1,\ldots, q,$ there exist an element $x\in {\mathcal
I}_i\setminus {\mathfrak m}{\mathcal I}_i,$ $x$ not contained in
any prime ideal in $\Sigma,$ and a positive integer $k_i$ such
that for all $r_i\geq k_i$ and all non-negative integers
$r_1,\dots,r_{i-1},r_{i+1},\dots,r_q$
$${\mathcal I}_1^{r_1}\cdots {\mathcal I}_i^{r_i}\cdots {\mathcal I}_q^{r_q}M\cap xM=x{\mathcal I}_1^{r_1}\cdots {\mathcal I}_i^{r_i-1}\cdots {\mathcal I}_q^{r_q}M.$$ \noindent Hence,
the result follows by taking degree $|{\bf r}|+p$ in the above
equality.
\end{proof}

The following result shows the existence of weak-(FC)-sequences.

\begin{prop}\label{Obs1}
If ${\mathcal I}$  is not contained in $\sqrt{\mbox{Ann}M}$ then,
for any $i=1,\ldots,q,$ there exists a weak-(FC)-element $x_i\in
I_i$
 with respect to $(I_1,...,I_q;M).$
\end{prop}

\begin{proof}
Set ${\Sigma}=\mbox{Ass}(\frac{M}{0_{M}:{\mathcal I}^{\infty}}).$
Since ${\mathcal I}$  is assumed not to be contained in
$\sqrt{\mbox{Ann}M}$ we have that ${\Sigma}\neq\emptyset $. We can
easily see that ${\Sigma}$ is a finite set and also that
${\Sigma}= \{\mathfrak{p}\in \mbox{Ass}(M)| {\mathfrak p}
\not\supseteq {\mathcal I}\}.$ By Lemma \ref{Rees Lemma}, for each
$i=1,\dots, s$ there exists  an element $x_i\in I_i\setminus
{\mathfrak m}I_i$ such that  $x_i$ satisfies the condition
$(FC_1)$ and $x_i \notin {\mathfrak p}$ for all ${\mathfrak p}\in
{\Sigma}$. Thus, $x_i$ also satisfies the condition $(FC_2)$.
Hence $x_i\in I_i$ is a weak-(FC)-element with respect to
$(I_1,...,I_q;M).$
\end{proof}

Vi\^et in \cite{Viet11} introduced the concept of generalized
joint reductions of ideals in local rings, which is a
generalization of the notion of joint reduction given by Rees in
\cite{Rees}. This notion was lately extended to module
coefficients by Manh and Vi\^et in \cite{Manh-Viet}. Now, we
extend the notion of generalized joint reduction with module
coefficients to graded modules.

\begin{defn}
Let $I_1,\dots, I_q$ be a finitely generated $R$-submodules of $G_1.$ A set $({\mathfrak J}_1,\dots, {\mathfrak J}_t)$ of $R$-submodules of $G_1,$
with $\mathfrak J_i\subseteq I_i,\;i=1,\dots,t\leq q,$
is called a generalized joint reduction of $(I_1,\dots,I_q)$ with respect to
$M$ if
$${\bold I}^{{\bf r}}M_p=\Sigma_{j=1}^{t}{\mathfrak I}_j{\bold I}^{{\bf r}-\delta(j)}M_p,$$  \noindent for all large $r_1,\dots, r_q$ and all $p.$
\end{defn}

The relation between maximal weak-(FC)-sequences and generalized
joint reductions was determined in \cite[Theorem 3.4]{Viet11} and
\cite[Theorem 2.9]{Manh-Viet} in local rings. We extend this
result to graded modules as follows.

\begin{prop}\label{FC JoiRe}
 Assume that
${\mathcal I}$ is not contained in $\sqrt{\mbox{Ann}M}$ and let
 $J$ be a finitely generated $R$-submodule of $G_1$ of finite
colength. Suppose
 $${\mathfrak I}_1=(x_{11},\dots,x_{1m})\subset I_1,$$ $${\mathfrak I}_2=
(x_{21},\dots,x_{2n})\subset I_2,$$ $$\dots$$
$${\mathfrak I}_t=(x_{t1},\dots,x_{tp}) \subset I_t$$
\noindent and $x_{11},\dots,x_{1m}, x_{21},\dots,x_{2n}\dots
x_{t1},\dots,x_{tp}$ is a maximal weak-(FC)-sequence of $G$ with
respect to $(I_1,...,I_q;M)$ in ${\cup}_{i=1}^qI_i.$ Then the
following statements hold.
\begin{itemize}
\item [(i)] For any $k\leq t$, we have
$$({\mathfrak I}_1,{\mathfrak I}_2,\dots {\mathfrak I}_k)M_{p+|\bold r|-1}
\cap {\bold I}^{{\bf r}}M_p=\Sigma_{j=1}^{k}{\mathfrak I}_j{\bold I}^{{\bf r}-\delta(j)}M_p,$$ for all large ${\bold r}$ and all $p.$

\item [(ii)] ${\mathfrak I}_1,{\mathfrak I}_2,\dots {\mathfrak I}_k$ is a generalized joint reduction of $(I_1,\dots,I_q)$ with respect to $M.$

\end{itemize}
\end{prop}
\begin{proof}
 We prove first $(i)$ by using induction on $k\leq t$. For $k=1$, we shall show that
$$(x_{11},\dots,x_{1i})M_{p+|{\bf r}-\delta(1)|}
\cap {\bold I}^{{\bf r}}M_p=(x_{11},\dots,x_{1i}){\bold I}^{{\bf r}-\delta(1)}M_p,$$ \noindent for all large ${\bold r},$ by
induction on $i\leq m$. For $i=0$, the result trivially holds. Set
$L=(x_{11},\dots,x_{1{i-1}})M:{\mathcal I}^{\infty}.$ Since
$x_{11},\dots,x_{1i}\in I_i$ is a weak-$FC$-sequence  with
respect to $(J,I_1,...,I_q;M),$
$$(x_{1i}M_{p+|{\bf r}-\delta(1)|}+L_{p+|{\bf r}|})
\cap ({\bold I}^{{\bf r}}M_p+L_{p+|{\bf r}|})= x_{1i}{\bold I}^{{\bf r}-\delta(1)}M_p+L_{p+|{\bf r}|},$$ \noindent for all
large ${\bold r}$ and all $p.$ Hence we have

$(x_{1i}M_{p+|{\bf r}-\delta(1)|}+L_{p+|{\bf r}|})
\cap {\bold I}^{{\bf r}}M_p$
$$
\begin{array}{lll}
\vspace{0.3cm}
&=&{\bold I}^{{\bf r}}M_p\cap ({\bold I}^{{\bf r}}M_p+
L_{p+|{\bf r}|})\cap (x_{1i}M_{p+|{\bf r}-\delta(1)|}+L_{p+|{\bf r}|})\\
\vspace{0.3cm}
&=&{\bold I}^{{\bf r}}M_p\cap (x_{1i}{\bold I}^{{\bf r}-\delta(1)}M_p+L_{p+|{\bf r}|})\\
\vspace{0.3cm}
&=&x_{1i}{\bold I}^{{\bf r}-\delta(1)}M_p+{\bold I}^{{\bf r}}M_p\cap L_{p+|{\bf r}|}
\end{array}
$$
\noindent for all large ${\bold r}$ and all $p.$ By Artin-Rees Lemma, for all large ${\bold r}$ and all $p,$ we
have
$${\bold I}^{{\bf r}}M_p\cap L_{p+|{\bf r}|}\subseteq {\bold I}^{{\bf r}}M_p\cap (x_{11},\dots,x_{1{i-1}})M_{p+|{\bf r}-\delta(1)|}$$
\noindent and hence by inductive assumption,
$${\bold I}^{{\bf r}}M_p\cap L_{p+|{\bf r}|}\subseteq (x_{11},\dots,x_{1{i-1}}){\bold I}^{{\bf r}-\delta(1)}M_{p}$$
\noindent for all large ${\bold r}$ and all $p.$ Therefore,
$${\bold I}^{{\bf r}}M_p\cap L_{p+|{\bf r}|}=(x_{11},\dots,x_{1{i-1}}){\bold I}^{{\bf r}-\delta(1)}M_p$$
\noindent for all large ${\bold r}$ and all $p.$ In short we get
$$
\begin{array}{lll}
(x_{1i}M_{p+|{\bf r}-\delta(1)|}+L_{p+|{\bf r}|})
\cap {\bold I}^{{\bf r}}M_p&=&x_{1i}{\bold I}^{{\bf r}-\delta(1)}M_p+(x_{11},\dots,x_{1{i-1}}){\bold I}^{{\bf r}-\delta(1)}M_p\\
&=&(x_{11},\dots,x_{1i}){\bold I}^{{\bf r}-\delta(1)}M_p
\end{array}
$$
\noindent for all large ${\bold r}$ and all $p.$ But now, by Artin-Rees Lemma again, we have
$$(x_{1i}M_{p+|{\bf r}-\delta(1)|}+L_{p+|{\bf r}|})
\cap {\bold I}^{{\bf r}}M_p=(x_{11},\dots,x_{1i})M_{p+|{\bf r}-\delta(1)|}\cap {\bold I}^{{\bf r}}M_p$$ \noindent for all
large ${\bold r}$ and all $p.$ Thus,
$$(x_{11},\dots,x_{1i})M_{p+|{\bf r}-\delta(1)|}\cap {\bold I}^{{\bf r}}M_p=
(x_{11},\dots,x_{1i}){\bold I}^{{\bf r}-\delta(1)}M_p$$ \noindent
for all large ${\bold r}$, all $p$ and $i\leq m.$ In particular we get that
$${\mathfrak I}_1M_{p+|{\bf r}-\delta(1)|}\cap {\bold I}^{{\bf r}}M_p={\mathfrak I}_1{\bold I}^{{\bf r}-\delta(1)}M_p$$
\noindent for all large ${\bold r}$ and all $p$. Thus the result is proved for
$k=1.$

Set ${\mathfrak N}=({\mathfrak J}_1,\dots, {\mathfrak
J}_{k-1})M:{\mathcal I}^{\infty}$. By Artin-Rees Lemma, we have
$${\mathfrak N}_{p+|{\bf r}|}\cap {\bold I}^{{\bf r}}M_p\subseteq {\bold I}^{{\bf r}}M_p
\cap ({\mathfrak J}_1,\dots, {\mathfrak J}_{k-1})M_{p+|{{\bf r}}|-1}$$ \noindent and
$$({\mathfrak N}_{p+|{\bf r}|}+{\mathfrak J}_kM_{p+|\bold r-\delta(k)|})\cap
 {\bold I}^{{\bf r}}M_p\subseteq {\bold I}^{{\bf r}}M_p
\cap ({\mathfrak J}_1,\dots, {\mathfrak J}_k)M_{p+|{{\bf r}}|-1}$$ \noindent for all large
${\bold r}$ and all $p$. Therefore
$${\mathfrak N}_{p+|{\bf r}|}\cap {\bold I}^{{\bf r}}M_p=
{\bold I}^{{\bf r}}M_p\cap ({\mathfrak J}_1,\dots, {\mathfrak
J}_{k-1})M_{p+|{{\bf r}}|-1}$$
\noindent and
\begin{equation}\label{equetion1}
({\mathfrak N}_{p+|{\bf r}|}+{\mathfrak J}_kM_{p+|\bold r-\delta(k)|})\cap {\bold I}^{{\bf r}}M_p=
{\bold I}^{{\bf r}}M_p\cap ({\mathfrak J}_1,\dots,{\mathfrak J}_k)M_{p+|{{\bf r}}|-1}
\end{equation}
\noindent for all large ${\bold r}$ and all $p$. Since the result holds for
$k=1$, we have
$$({\mathfrak N}_{p+|{\bf r}|}+{\mathfrak J}_kM_{p+|\bold r-\delta(k)|})
\cap ({\bold I}^{{\bf r}}M_p+{\mathfrak N}_{p+|{\bf r}|})=
{\mathfrak J}_k{\bold I}^{\bold r-\delta(k)}M_{p}+{\mathfrak N}_{p+|{\bf r}|}$$ \noindent for all large ${\bold r}$ and all $p$. Putting all this facts together we get
$$
\begin{array}{l}
\vspace{0.3cm}
({\mathfrak N}_{p+|{\bf r}|}+{\mathfrak J}_kM_{p+|\bold r-\delta(k)|})\cap {\bold I}^{{\bf r}}M_p\\
\vspace{0.3cm}
 = {\bold I}^{{\bf r}}M_p\cap ({\mathfrak N}_{p+|{\bf r}|}+{\mathfrak
J}_kM_{p+|\bold r-\delta(k)|})\cap ({\bold I}^{{\bf r}}M_p+
{\mathfrak N}_{p+|{\bf r}|})\\
\vspace{0.3cm}
={\bold I}^{{\bf r}}M_p\cap({\mathfrak J}_k{\bold I}^{\bold r-\delta(k)}M_{p}+{\mathfrak N}_{p+|{\bf r}|})\\
\vspace{0.3cm}
={\mathfrak J}_k{\bold I}^{\bold r-\delta(k)}M_{p}+{\mathfrak N}_{p+|{\bf r}|}\cap {\bold I}^{{\bf r}}M_p\\
\vspace{0.3cm}
={\mathfrak J}_k{\bold I}^{\bold r-\delta(k)}M_{p}+({\mathfrak
J}_1,\dots, {\mathfrak J}_{k-1})M_{p+|\bold r|-1}\cap {\bold I}^{{\bf r}}M_p\,\,\,\,\,\,\,\,\,\,\,\,\,\,\,\,\,\,\,\,\,\,\,\,\,\,\,\,\,\,\,\,\, (*)
\end{array}
$$
\noindent for all large ${\bold r}$ and all $p$. But now, by inductive assumption we see that
\begin{equation}\label{equetion3}
({\mathfrak J}_1,\dots, {\mathfrak J}_{k-1})M_{p+|\bold r|-1}\cap {\bold I}^{{\bf r}}M_p=\sum_{j=1}^{k-1}
{\mathfrak J}_j{\bold I}^{\bold r-\delta(j)}M_{p}
\end{equation}
\noindent for all large ${\bold r}$ and all $p$. Hence by (\ref{equetion1}),
$(*)$ and (\ref{equetion3}), we get
$$({\mathfrak J}_1,\dots, {\mathfrak J}_k)M_{p+|\bold r|-
1}\cap {\bold I}^{{\bf r}}M_p=\sum_{j=1}^k{\mathfrak J}_j{\bold I}^{\bold r-\delta(j)}M_{p}$$
\noindent for all large ${\bold r}$ and all $p$.

We now prove $(ii)$. Since $x_{11},\dots,x_{1m},
x_{21},\dots,x_{2n}\dots x_{t1},\dots,x_{tp}$ is a maximal
weak-(FC)-sequence  with respect to $(I_1,\dots,I_q;M)$ in
${\cup}_{i=1}^qI_i,$  by Proposition \ref{Obs1}, we have that
$${\bold I}^{{\bf r}}M_p\subset ({\mathfrak J}_1,\dots, {\mathfrak J}_t)M_{p+|\bold r|-1}$$ \noindent for all large
${\bold r}$ and all $p$. Hence the result follows by using part $(i).$
\end{proof}

The following result is also an immediate consequence of
Proposition \ref{FC JoiRe}.

\begin{cor}
Suppose that $\dim(\mbox{Supp}(M))=\dim(\mbox{Proj}(G))=r$ and
that $I_1,\dots, I_r$ are $R$-submodules of $G_1$ of finite
colength. Assume that $x_1,\dots, x_r$ is a weak-(FC)-sequence
with respect to $(I_1,\dots, I_r;M),$ with $x_i\in I_i$ for
$i=1,\dots,r.$ Then $x_1,\dots, x_r$ is a joint reduction of
$I_1,\dots, I_r$ with respect to $M.$
\end{cor}

Since the existence of maximal-weak-(FC)-sequences is secured by
Proposition \ref{Obs1}, the above corollary gives us the existence
of joint reductions for a family of $R$-submodules $I_1,\dots,
I_r$ of $G_1$ of finite colength. This result was also obtained in
\cite[Proposition 3.7 and Remark 3.8]{Bedregal-Perez}.

\section{Buchsbaum-Rim multiplicities}

Fix $(R, \mathfrak m)$ an arbitrary
Noetherian local ring; fix two graded $R$-algebras $G'=\oplus
G'_n$ and $G=\oplus G_n,$ with $G'\subseteq G,$ that, as usual,
are generated as algebras by finitely many elements of degree one
such that $\ell(G_1/G'_1)<\infty;$ and fix $M$ a finitely
generated graded $G$-module. Let $r:=\dim(\mbox{Proj}(G))$ be the
dimension of $\mbox{Proj}(G).$ As a function of $n,q$, the length,
$$h(n,q):=\ell(M_{n+q}/G'_nM_q)$$
\noindent is eventually a polynomial in $n,q$, denoted by
$P_{G',G,M}(n,q),$ of total degree equal to
$\dim(\mbox{Supp}(M)),$ which is at most $r,$ (see \cite[Theorem
5.7]{Kleiman-Thorup}) and the coefficient of $n^{r-j}q^j/(r-j)!j!$
is denoted by $e^j(G',G,M),$ for all $j=0,\ldots, r,$  and it is
called the $j^{\mbox{th}}$  {\it associated Buchsbaum-Rim
multiplicity} of the pair $(G',G)$  with respect to $M.$
Notice that $e^j(G',G,M)=0$ if $\dim(\mbox{Supp}(M))<r.$ The
number $e^0(G',G,M)$ will be called the {\it Buchsbaum-Rim
multiplicity} of the pair $(G',G)$  with respect to $M,$ and
will also be denoted by $e_{BR}(G',G,M).$ The notion of
Buchsbaum-Rim multiplicity for modules goes back to
\cite{Buchsbaum-Rim} and it was carried out in the above
generality in
\cite{Kleiman-Thorup}, \cite{Kleiman-Thorup2}, \cite{Kirby}, \cite{Kirby-Rees1},  \cite{Kirby-Rees2}, \cite{Katz}, \cite{Bedregal-Perez}
and \cite{Simis-Ulrich-Vasconcelos}.

\begin{rem}\label{remark0}{\rm
If $I$ is a finitely generated $R$-submodule of $G_1$ such that
$\ell(G_1/I)<\infty$ then, setting $G'=R[I],$ the $R$-subalgebra
of $G$ generated in degree one by $I,$ we denote $e^j(G',G,M)$
(resp.  $e_{BR}(G',G,M)$) by $e^j(I,M)$ (resp. $e_{BR}(I,M)$),
which is called the {\it $j^{\mbox{th}}$-Associated Buchsbaum-Rim
multiplicity} (resp. {\it Buchsbaum-Rim multiplicity} ) {\it of
the pair} $(I,G)$ {\it with respect to} $M.$ }
\end{rem}

\section{Mixed multiplicities}\label{Section}

We keep the notations of Setup $(1)$. In this section we define
the notion of mixed multiplicities of $J,I_1,\dots, I_q$ with
respect to $M.$ The main results of this section establish mixed
multiplicity formulas by means of Buchsbaum-Rim multiplicities and
also determines the positivity of such mixed multiplicities.

Consider the function
$$h(n,p,{\bold r}):=\ell \left(\frac{{\bf I}^{\bf r}M_{n+p}}
{J^{n}{\bf I}^{\bf r}M_{p}}\right ).$$

\begin{thm}\label{prop1} Keeping the Setup $(1)$, set
$D=\dim(\mbox{Supp}(M^*)).$ Then $h(n,p,{\bold r})$ is a
polynomial of degree $D$ for all large $n,p,{\bold r}.$
\end{thm}

\begin{proof}
We will prove first the following equality
\begin{equation}\label{equation1}\ell\left(\frac{{\bf I}^{\bf r}M_{n+p}}{J^{n}{\bf I}^{\bf r}M_{p}}\right )
=\ell\left(\frac{{\bf I}^{\bf r}M^*_{n+p}}
{J^{n}{\bf I}^{\bf r}M^*_{p}}\right )\end{equation}

Notice that
$$\begin{array}{lll} \vspace{0.3cm}\ell\left(\frac{{\bf I}^{\bf r}M^*_{n+p}}{J^{n}{\bf I}^{\bf r}M^*_{p}}\right) & = &
\ell\left(\frac{N_{n+|{\bf r}|+p}+{\bf I}^{\bf r}M_{n+p}}{N_{n+|{\bf r}|+p}+J^{n}{\bf I}^{\bf r}M_{p}}\right )\\
\vspace{0.3cm} & = & \ell\left(\frac{{\bf I}^{\bf r}M_{n+p}}{J^{n}{\bf I}^{\bf r}M_{p}+N_{n+|{\bf r}|+p}\cap {\bf I}^{\bf r}M_{n+p}}\right).
\end{array}$$

By Artin-Rees Lemma, there exist integers $n_0, p_0, c_1,\ldots,
c_q\in \N$ such that
$$N_{n+|{\bf r}|+p}\cap {\bf I}^{\bf r}M_{n+p}={\bf I}^{{\bf r}-{\bf c}}(N_{n+p+|{\bold
c}|}\cap {\bf I}^{\bf c}M_{n+p}),$$
\noindent for all $p\geq p_0, n\geq n_0, {\bold r}\geq {\bold c}.$
Since for all large $t$  we have $N{\mathcal I}^t=0,$ it follows that
$$N_{n+|{\bf r}|+p}\cap {\bf I}^{\bf r}M_{n+p}=0$$
\noindent and hence we prove equality
(\ref{equation1}).

Let $$K=\bigoplus_{n,p,{\bf r}}\frac{{\bf I}^{\bf r}M^*_{n+p}}
{J^{n}{\bf I}^{\bf r}M^*_{p}}$$ and $$h^*(n,p,{\bf r})=\ell \left(\frac{{\bf I}^{\bf r}M^*_{n+p}}
{J^{n}{\bf I}^{\bf r}M^*_{p}}\right )$$
We want to prove that $h^*$ is a polynomial of degree $\dim(\mbox{Supp}(M^*)).$
Consider $H=\oplus_{n}J^n,$ $H^i=\oplus_{r_i}I_i^{r_i}$ and $N=\oplus_{n,p,{\bf r}}\frac{J^n{\bf I}^{\bf r}M_p^*}{J^{n+1}{\bf I}^{\bf r}M_{p-1}^*}$,
 it is easy to see that $N$ is a multigraded $H\otimes G\otimes H^1\otimes\cdots \otimes H^q$-module,
 by \cite[Lemma 4.3]{Kleiman-Thorup} we have that  $$\ell \left(\frac{J^n{\bf I}^{\bf r}M_p^*}{J^{n+1}{\bf I}^{\bf r}M_{p-1}^*}\right)$$ is a polynomial
of degree $\dim(\mbox{Supp}(N)))$. Therefore $h^*(n,p,{\bf r})$ is a polynomial of degree $\dim(\mbox{Supp}(K))=\dim(\mbox{Supp}(N))+1.$

On the other hand  take $r_1=\ldots
=r_q=u$ and fix the $u$ and let $L=\oplus_{n}{\bf I}^{\bf u}M_n^*=\oplus_{n}L_n.$ Define $h_L(n,p)=\ell \left(\frac{L_{n+p}}{J^nL_p}\right)$, by \cite[Theorem 5.7]{Kleiman-Thorup} we have that $h_L(n,p)$ is a polynomial of degree $\dim(\mbox{Supp}(L))$.

Now
consider the exact sequence
\begin{equation}\label{equation5}0\longrightarrow
L \longrightarrow M^* \longrightarrow
M^*/L\longrightarrow 0.
\end{equation}
\noindent Since
$\mathcal I$ is not contained $\sqrt{\mbox{Ann}M},$ there exist an
element $x\in I_1^u\cdots I_q^u$ which is not a zero-divisor of
$M^*.$ Thus

\begin{equation}\label{equation4}\mbox{dim}(\mbox{Supp}(L)=\mbox{dim}(\mbox{Supp}(M^*))>\mbox{dim}(\mbox{Supp}(M^*/L)).
\end{equation}
We have that $\dim(\mbox{Supp}(L))=\dim(\mbox{Supp}(M^*)).$

As
$h_L(n,p)=h^*(n,p,{\bf u})$ we have $\deg(h_L(n,p))=\deg(h^*(n,p,{\bf r}))$ therefore $h^*(n,p,{\bf r})$ is a polynomial of degree $\dim(\mbox{Supp}(M^*)).$

Now, $$h^*(n,p, {\bf
r}):=\ell\left(\frac{{\bf I}^{\bf r}M^*_{n+p}} {J^{n}{\bf I}^{\bf
r}M^*_{p}}\right)$$ is a polynomial of degree $D$, where
$D:=\dim(\mbox{Supp}(M^{*}))$ for all large $n,p,{\bf r}.$
Therefore

$$\ell\left(\frac{{\bf I}^{\bf r}M_{n+p}}{J^{n}{\bf I}^{\bf r}M_{p}}\right)$$ \noindent is a polynomial of degree $D$ for all
large ${\bf r}.$
\end{proof}

If we write the terms of total degree
$D=\dim(\mbox{Supp}(M^{*}))$ of the polynomial
$$\ell\left(\frac{{\bf I}^{\bf r}M_{n+p}}{J^{n}{\bf I}^{\bf r}M_{p}}\right)$$ \noindent in the
form
$$B(n,p,{\bf r})=\sum_{k_0+|{\bf k}|+j=D}\;\frac{1}{k_0!{\bf k}!j!}e^j(J^{[k_0]},I_1^{[k_1]},\dots,I_q^{[k_q]};M)
n^{k_0}p^j{\bf r}^{\bf k}.$$ \noindent The coefficients
$e^j(J^{[k_0]},I_1^{[k_1]},\dots,I_q^{[k_q]};M)$ are called the
$j^{\mbox th}$-{\it mixed multiplicity} of $(J, I_1,\ldots, I_q;
M)$. We call $e^0(J^{[k_0]},I_1^{[k_1]},\dots,I_q^{[k_q]};M)$ the
{\it mixed multiplicity} of $(J, I_1,\ldots, I_q; M)$ of type
$(k_0,k_1,\dots,k_q)$.

\begin{rem}\label{remark1}
{\rm It follows by the equality (\ref{equation1}) that
$$e^j(J^{[k_0]},I_1^{[k_1]},\dots,I_q^{[k_q]};M)=e^j(J^{[k_0]},I_1^{[k_1]},\dots,I_q^{[k_q]};M^*)$$
\noindent for all $k_0+|{\bf k}|+j=D.$
}\end{rem}

\begin{lem}\label{lema crucial} Keeping the setup $(1)$ we have
 \begin{itemize}
\item [(i)] $e^j(J^{[k_0]},I_1^{[0]},\dots,I_q^{[0]};M)\!\neq \!0$ and
$e^j(J^{[k_0]},I_1^{[0]},\dots,I_q^{[0]};M)\!\!=\!\!e^j_{BR}(J;M^*).$

\item [(ii)] $e^j(J^{[k_0]},I_1^{[0]},\dots,I_q^{[0]};M)=e^j_{BR}(J;M)$ if
$\mbox{ht}(\mathcal{I}+\mbox{Ann}M/\mbox{Ann}M)>0.$
\end{itemize}
\end{lem}

\begin{proof}
We prove first $(i).$ By Proposition \ref{prop1} there exist a
positive integer $u$ such that
$$\ell\left(\frac{{\bf I}^{\bf r}M_{n+p}}{J^{n}{\bf I}^{\bf r}M_{p}}\right)=\ell\left(\frac{{\bf I}^{\bf r}M^*_{n+p}}
{J^{n}{\bf I}^{\bf r}M^*_{p}}\right)$$ \noindent is a polynomial
of degree $D$ for all $n,p,r_1,\ldots, r_q\geq u.$ We shall denote
by $P(n,p, {\bold r})$ the above polynomial.  Take $r_1=\ldots
=r_q=u$ and fix the $u.$ Set $H(n,p)=P(n,p,u,\ldots, u).$ Then we
get
$$H(n,p)=\sum_{k_0+j=D}\;\frac{1}{k_0!j!}\;e^j(J^{[k_0]},I_1^{[0]},\dots,I_q^{[0]};M)
p^jn^{k_0} + ...$$

\noindent Now set  $L=I_1^u\cdots I_q^uM^*$.

 Hence, we have by $(\ref{equation4})$ and \cite[Theorem
(5.7)]{Kleiman-Thorup} that the function
$\varphi(n,p)\!=\!\ell\left(\frac{L_{n+p}}
{J^{n}L_p}\right)$ is eventually a polynomial of
degree at most $D=\mbox{dim}(\mbox{Supp}(M^*)).$ Moreover, for all
$n,p\gg 0$ we have that
$$\varphi(n,p)=\sum_{k_0+j=D}\;\frac{1}{k_0!j!}\;e^j_{BR}(J,L)
p^jn^{k_0} + ...$$
\noindent But clearly
$\varphi(n,p)=P(n,p,u,\ldots, u)$ and hence
\begin{equation}\label{equation15} e^j(J^{[k_0]},I_1^{[0]},\dots,I_q^{[0]};M)=e^j_{BR}(J,L)\neq 0.\end{equation}

\noindent On the other hand, by the additivity property of the
associated Buchsbaum-Rim multiplicities (see \cite[Theorem
(6.7a)]{Kleiman-Thorup}) and (\ref{equation5}) and
(\ref{equation4}), we have that
\begin{equation}\label{equation14} e^j_{BR}(J,L)=e^j_{BR}(J,M^*).\end{equation}
Now the result follows by (\ref{equation15}) and
(\ref{equation14}).

We now prove $(ii)$. Consider the exact sequence
$$0\longrightarrow (0_M:{\mathcal I}^{\infty})\longrightarrow M
\longrightarrow M^*\longrightarrow 0.$$ We can easily show that
$$\mbox{Ass}_{G}(M)=\mbox{Ass}_{G}(0_M:{\mathcal I}^{\infty})\cup
\mbox{Ass}_{G}(M^*)$$ \noindent and $$\mbox{Ass}_{G}(0_M:{\mathcal
I}^{\infty})\cap \mbox{Ass}_{G}(M^*)=\emptyset.$$ Hence, since
also $\mbox{ht}(\mathcal{I}+\mbox{Ann}M/\mbox{Ann}M)>0$, any prime
ideal $\mathfrak p\in \mbox{Ass}_{G}(0_M:{\mathcal I}^{\infty})$
is a non-minimal element in $\mbox{Ass}_{G}(M).$ From this follows
that
\begin{equation}\label{equation6}\mbox{dim}(\mbox{Supp}((0_M:{\mathcal
I}^{\infty}))< \mbox{dim}(\mbox{Supp}(M))=
\mbox{dim}(\mbox{Supp}(M^*)).\end{equation}

Therefore, by the additivity property of the associated
Buchsbaum-Rim multiplicities (see \cite[Theorem (6.7a)]{Kleiman-Thorup}), we have that
$e^j_{BR}(J,M^*)=e^j_{BR}(J,M).$ Hence by $(i)$, we get
$$e^j(J^{[k_0]},I_1^{[0]},\dots,I_q^{[0]};M)=e^j_{BR}(J,M^*)=
e^j_{BR}(J,M).$$
\end{proof}

The following proposition plays a crucial role for establishing
the relation between associated mixed multiplicities and
Buchsbaum-Rim multiplicities.

\begin{prop}\label{prop2}
Keeping the setup $(1)$,  the following statements hold.
\begin{itemize}
\item [(i)] If $x\in I_i,\;i\geq 1,$ is an (FC)-element  with respect to
$(J,I_1,...,I_q;M)$, then  $$e^j(J^{[k_0]},I_1^{[k_1]},\dots,I_q^{[k_q]};M)=e^j(J^{[k_0]},I_1^{[k_1]},\dots, I_i^{[k_i-1]},\dots,I_q^{[k_q]};\overline{M}),$$ \noindent where $k_i$ is a positive
integer and $\overline{M}:=M/xM.$
\item [(ii)] If $e^j(J^{[k_0]},I_1^{[k_1]},\dots,I_q^{[k_q]};M)\neq 0,$ then for any $i\geq 1$
such that $k_i>0,$ there exists an (FC)-element $x\in I_i$ with respect to
$(J,I_1,...,I_q;M)$.
\end{itemize}
\end{prop}

\begin{proof}
By Proposition \ref{prop1} we know that the function
$$\ell\left(\frac{{\bf I}^{\bf r}M_{n+p}}{J^{n}{\bf I}^{\bf r}M_{p}}\right).$$
is, for large $n,p,\bold r,$ a polynomial function, which we
denote by $B(n,p,\bold r)$, of degree at most
$D=\dim(\mbox{Supp}(M^*)).$
 Set
$\overline{M}=M/xM$, ${\overline{M}}^*:=M/xM:{\mathcal
I}^{\infty},$ it is clear that ${\overline{M}}^*\simeq
\frac{\overline{M}}{0_{\overline M}:{\mathcal I}^{\infty}}.$
\noindent Then, for all $n\gg 0$ and $\bold r\gg \bold 0$, we have
$$
\begin{array}{lll}
\vspace{0.3cm}
\ell \left(\frac{{\bf I}^{\bf r}{\overline{M}}^*_p}
{J^n{\bf I}^{\bf r}{\overline{M}}^*_p}\right )&=&\ell\left(\frac{{\bf I}^{\bf r}M^*_{n+p}+xM^*_{|{\bold r}|+n+p-1}}{J^n{\bf I}^{\bf r}M^*_p+xM^*_{n+|{\bold r}|+p-1}}\right)\\
\vspace{0.3cm}
&=&\ell\left(\frac{{\bf I}^{\bf r}M^*_{n+p}}{J^n{\bf I}^{\bf r}M^*_p+xM^*_{n+|{\bold r}|+p-1}\cap {\bf I}^{\bf r}M^*_{n+p}}\right)\\
\vspace{0.3cm}
&=&\ell\left(\frac{{\bf I}^{\bf r}M^*_{n+p}}{J^n{\bf I}^{\bf r}M^*_p+x{\bf I}^{{\bf r}-\delta(i)}M^*_{n+p}}\right)\\
\vspace{0.3cm}
&=&\ell\left(\frac{{\bf I}^{\bf r}M^*_{n+p}}{J^n{\bf I}^{\bf r}M^*_p}\right)-\ell\left(\frac{J^n{\bf I}^{\bf r}M^*_p+x{\bf I}^{{\bf r}-\delta(i)}M^*_{n+p}}{J^n{\bf I}^{\bf r}M^*_p}\right)\\
\vspace{0.3cm}
&=&\ell\left(\frac{{\bf I}^{\bf r}M^*_{n+p}}{J^n{\bf I}^{\bf r}M^*_p}\right)-\ell\left(\frac{x{\bf I}^{{\bf r}-\delta(i)}M^*_{n+p}}{J^n{\bf I}^{\bf r}M^*_p\cap xM^*_{n+p+|{\bold r}|}\cap {\bf I}^{\bf r}M^*_p}\right)\\
\vspace{0.3cm}
&=&\ell\left(\frac{{\bf I}^{\bf r}M^*_{n+p}}{J^n{\bf I}^{\bf r}M^*_p}\right)-\ell\left(\frac{x{\bf I}^{{\bf r}-\delta(i)}M^*_{n+p}}{xJ^n{\bf I}^{{\bf r}-\delta(i)}M^*_p}\right)
\end{array}
$$

Since $x\in I_i,\;i\geq 1,$ satisfies the condition $(FC_2)$ with
respect to {\break}$(J,I_1,...,I_q;M)$, it is a non-zero divisor of $M^*$
and hence we have an isomorphism of $R$-modules

$$\frac{x{\bf I}^{{\bf r}-\delta(i)}M^*_{n+p}}{xJ^n{\bf I}^{{\bf r}-\delta(i)}M^*_p}
\cong \frac{{\bf I}^{{\bf r}-\delta(i)}M^*_{n+p}}{J^n{\bf I}^{{\bf r}-\delta(i)}M^*_p}.$$ \noindent for all large $p,n, {\bold r}.$ So

$$\ell\left(\frac{x{\bf I}^{{\bf r}-\delta(i)}M^*_{n+p}}{xJ^n{\bf I}^{{\bf r}-\delta(i)}M^*_p}\right)=
\ell\left(\frac{{\bf I}^{{\bf r}-\delta(i)}M^*_{n+p}}{J^n{\bf I}^{{\bf r}-\delta(i)}M^*_p}\right),$$ \noindent for all $n,p$ and all large
$\bold r$.

\noindent Hence
\begin{equation}\label{equation17}\ell \left(\frac{{\bf I}^{\bf r}{\overline{M}}^*_{n+p}}
{J^n{\bf I}^{\bf r}{\overline{M}}^*_p}\right
)=\ell\left(\frac{{\bf I}^{\bf r}M^*_{n+p}}{J^n{\bf I}^{\bf
r}M^*_p}\right)-\ell\left(\frac{{\bf I}^{{\bf
r}-\delta(i)}M^*_{n+p}}{J^n{\bf I}^{{\bf
r}-\delta(i)}M^*_p}\right).\end{equation}

\noindent Now, by equation (\ref{equation1}),  we also have
$$\ell \left(\frac{{\bf I}^{\bf r}{\overline{M}}^*_{n+p}}
{J^n{\bf I}^{\bf r}{\overline{M}}^*_p}\right )=\ell \left(\frac{{\bf I}^{\bf r}{\overline{M}}_{n+p}}
{J^n{\bf I}^{\bf r}{\overline{M}}_p}\right )$$ \noindent for all large $n, p,{\bf r}$. Therefore
$$\ell \left(\frac{{\bf I}^{\bf r}{\overline{M}}_{n+p}} {J^n{\bf
I}^{\bf r}{\overline{M}}_p}\right )=\ell \left(\frac{{\bf I}^{\bf
r}M_{n+p}} {J^n{\bf I}^{\bf r}M_p}\right )-\ell \left(\frac{{\bf
I}^{{\bf r}-\delta(i)}M_{n+p}} {J^n{\bf I}^{{\bf
r}-\delta(i)}M_p}\right ),$$ \noindent for all large $p,{\bf r}$.

\noindent  Now if $x$ satisfies the condition $(FC_3)$, that means
$$\mbox{dim}(\mbox{Supp}(\overline{M}^*))={\mbox{dim}}(\mbox{Supp}(({M}^*)))-1$$
\noindent  then it can be verified that
$$\ell \left(\frac{{\bf I}^{\bf r}{\overline{M}}_{n+p}}
{J^n{\bf I}^{\bf r}{\overline{M}}_p}\right ).$$

\noindent is a polynomial of degree $D-1$ for all large $p,n,\bold r,$
and the terms of total degree $D-1$ in this polynomial

$$\ell \left(\frac{{\bf I}^{\bf r}{\overline{M}}_p}
{J^n{\bf I}^{\bf r}{\overline{M}}_p}\right ),$$

\noindent is equal to the terms of total degree $D-1$ in the
polynomial

$$\ell \left(\frac{{\bf I}^{\bf r}M_{n+p}}
{J^n{\bf I}^{\bf r}M_p}\right )-\ell \left(\frac{{\bf I}^{{\bf r}-\delta(i)}M_{n+p}}
{J^n{\bf I}^{{\bf r}-\delta(i)}M_p}\right ).$$

The above facts show that
$$e^j(J^{[k_0]},I_1^{[k_1]},\dots,I_q^{[k_q]};M)=e^j(J^{[k_0]},I_1^{[k_1]},\dots, I_i^{[k_i-1]},\dots,I_q^{[k_q]};\overline{M}).$$

We now prove $(ii)$. As ${\mathcal I}$ is not contained in
$\sqrt{\mbox{Ann}M}$ it follows by Proposition \ref{Obs1} that,
for any $i\geq 1$ such that $k_i>0,$ there exists an element $x\in
I_i$ which  is a weak-(FC)-element  with respect to
$(J,I_1,...,I_q;M)$. Now, by (\ref{equation17}) we have

$$\ell \left(\frac{{\bf I}^{\bf r}{\overline{M}}^*_{n+p}}
{J^n{\bf I}^{\bf r}{\overline{M}}^*_p}\right )=\ell\left(\frac{{\bf I}^{\bf r}M^*_{n+p}}{J^n{\bf I}^{\bf r}M^*_p}\right)-\ell\left(\frac{{\bf I}^{{\bf r}-\delta(i)}M^*_{n+p}}{J^n{\bf I}^{{\bf r}-\delta(i)}M^*_p}\right).$$

\noindent Since $e^j(J^{[k_0]},I_1^{[k_1]},\dots,I_q^{[k_q]};M)\neq 0,$ it follows that
$\ell \left(\frac{{\bf I}^{\bf r}{\overline{M}}^*_{n+p}}
{J^n{\bf I}^{\bf r}{\overline{M}}^*_p}\right )$
is a polynomial of degree $D-1$. Thus,
$$\dim(\mbox{Supp}({\overline{M}}^*))=\dim(\mbox{Supp}(M/xM:{\mathcal I}^{\infty}))=D-1=\dim(\mbox{Supp}((M^*))-1$$
and hence $x$ is an (FC)-element.
\end{proof}

\begin{prop}\label{prop3}
Let $x_1,\ldots, x_t$ be a weak-(FC)-sequence in $I_1\cup\dots\cup
I_q$ with respect to $(J, I_1,\dots, I_q;M)$. Then
$$\dim\left(\mbox{Supp} \left(M/(x_1,\ldots, x_t)M:\mathcal{I}^{\infty}\right)\right)\leq \dim\left(\mbox{Supp}
\left(M/0_{M}:\mathcal{I}^{\infty}\right)\right)-t$$ with equality
if and only if $x_1,\ldots, x_t$ is an (FC)-sequence of $G$ with
respect to $U.$
\end{prop}

\begin{proof}By equality (\ref{equation17}) it follows that
$$\ell\left(\frac{{\bf I}^{\bf r}\overline{M}^*_{n+p}}
{J^{n}{\bf I}^{\bf r}\overline{M}^*_p}\right)$$ \noindent is a
polynomial of degree at most $D-1$ for all $n>z$ and $\bold
r>\bold v.$ Hence $\dim(\mbox{Supp}(M/xM:{\mathcal
I}^{\infty}))\leq\dim(\mbox{Supp}(M/0:{\mathcal I}^{\infty}))-1.$
From this fact and using induction on $t$ it follows that
$$\dim(\mbox{Supp}(M/(x_1,\ldots, x_t)M:{\mathcal I}^{\infty}))\leq\dim(\mbox{Supp}
(M/0:{\mathcal I}^{\infty}))-t$$ and $x_1,\ldots, x_t$ is an
(FC)-sequence in $I_1\cup\dots \cup I_q$ with respect to {\break}$(J,
I_1,\dots, I_q;M)$ if and only if
$$\dim(\mbox{Supp}(M/(x_1,\ldots, x_t)M:{\mathcal I}^{\infty}))=\dim(\mbox{Supp}
(M/0:{\mathcal I}^{\infty}))-t.$$
\end{proof}

The following theorem characterizes mixed multiplicities  by means
of Buchsbaum-Rim multiplicities and also determines the positivity
of mixed multiplicities.

\begin{thm}\label{Teo1}
 Keeping the setup $(1),$ assume that $D>0.$
Let $k_0,j,k_1,...,k_q$ be non-negative integers with sum equal to
$D.$  Then
\begin{itemize}
\item [(i)] $$e^j(J^{[k_0]},I_1^{[k_1]},\dots, I_q^{[k_q]}; M)=e^j_{BR}({J};\overline{M}_t^*),$$
\noindent for any (FC)-sequence $x_1,...,x_t,$ with respect to
$(J,I_1,\ldots,I_q; M),$ of $t=k_1+...+k_q$ elements  consisting
of $k_1$ elements of $I_1$,..., $k_q$ elements of $I_q,$ where
$\overline{M}_t^*=M/((x_1,...,x_t)M:{\mathcal{I}}^{\infty}).$

\item [(ii)] If $k_0>0,$ then $e^j(J^{[k_0]},I_1^{[k_1]},\dots, I_q^{[k_q]};M)\neq 0,$
if and only if there exists a (FC)-sequence, with respect to
$(J,I_1,\ldots,I_q; M),$ of $t=k_1+...+k_q$ elements consisting of
$k_1$ elements of $I_1$,..., $k_q$ elements of $I_q.$
\end{itemize}
\end{thm}
\begin{proof}  We prove first $(i).$ Let $x_1,...,x_t$ be an
(FC)-sequence with respect to $(J,I_1,\ldots,I_q; M),$ of
$t=k_1+...+k_q$ elements  consisting of $k_1$ elements of
$I_1$,..., $k_q$ elements of $I_q.$ We shall begin by showing that
\begin{equation}\label{equation10}
e^j(J^{[k_0]},I_1^{[k_1]},\dots, I_q^{[k_q]};M)=e^j(J^{[k_0]},I_1^{[0]},\dots, I_q^{[0]};\overline{M_t})\neq 0,
\end{equation}
\noindent where  $\overline{M_t}=M/(x_1,...,x_t)M.$  The proof is by
induction on $t=k_1+...+k_q$. \noindent For $t=0$, since ${\mathcal
I}$ is not contained in $\sqrt{\mbox{Ann}M}$, by Lemma \ref{lema
crucial}

$$e^j(J^{[k_0]},I_1^{[k_1]},\dots, I_q^{[k_q]};M)=e^j(J^{[k_0]},I_1^{[0]},\dots, I_q^{[0]};{M})\neq 0,$$

\noindent hence the result holds.

Now, assume that  $t>0.$ Then there
 exists $i$ $(1\leq i \leq q)$ such that $k_i> 0$ and $x_1\in I_i$. By Proposition \ref{prop2}, we get
$$e^j(J^{[k_0]},I_1^{[k_1]},\dots, I_q^{[k_q]};M)=e^j(J^{[k_0]},I_1^{[k_1]},\dots,I_i^{[k_i-1]},\dots, I_q^{[k_q]};\overline{M_1}),$$
\noindent where  $\overline{M_1}:=M/x_1M.$ Since $x_1,\dots, x_t$
is an (FC)-sequence with respect to $(J,I_1,\dots,I_q;M),$ it
follows that $x_2,\dots, x_t$ is an (FC)-sequence with respect to
$(J,I_1,\dots,I_q;M_1).$ But $k_1+\dots +(k_i-1)+\dots +k_q=t-1$,
then it follows by inductive assumption
$$e^j(J^{[k_0]},I_1^{[k_1]},\dots,I_i^{[k_i-1]},\dots, I_q^{[k_q]};\overline{M_1})=e^j(J^{[k_0]},I_1^{[0]},\dots, I_q^{[0]};\overline{M_t})\neq
0.$$ The induction is complete. We now turn to the proof of $(i)$.
Notice that since  $\overline{M_t}=M/(x_1,...,x_t)M$ we have that
\begin{equation}\label{equation7}\overline{M_t}^*=\overline{M_t}/(0_{\overline{M_t}}:{\mathcal{I}}^{\infty})=M/((x_1,...,x_t)M:{\mathcal{I}}^{\infty}).\end{equation}
 Since
$e^j(J^{[k_0]},I_1^{[0]},\dots, I_q^{[0]};\overline{M_t})\neq 0$, it follows that
$\mathcal I$ is not contained in
$\sqrt{\mbox{Ann}\overline{M_t}}$. Now, from Lemma \ref{lema crucial}
we have
$$e^j(J^{[k_0]},I_1^{[0]},\dots, I_q^{[0]};\overline{M_t})=
e_{BR}^j(J,{\overline{M_t}^*}).$$
\noindent Hence, by the  equality (\ref{equation10}),
$$e^j(J^{[k_0]},I_1^{[k_1]},\dots, I_q^{[k_q]};M)=e_{BR}^j(J,{\overline{M_t}^*}),$$ \noindent and the proof of
$(i)$ is complete.

We now prove $(ii).$ We prove first the necessity. The proof is by
induction on $t=k_1+...+k_q$. If $t=0$, the result is trivial. So,
assume that $t>0,$ that is some $k_i>0.$ Since by assumption
$e^j(J^{[k_0]},I_1^{[k_1]},\dots, I_q^{[k_q]};M)\neq 0,$ by
Proposition \ref{prop2}, there exists an element $x_1\in I_i$
which is an  (FC)-element with respect to $(J,I_1,\dots,I_q;M)$
 and $$e^j(J^{[k_0]},I_1^{[k_1]},\dots, I_q^{[k_q]};M)=e^j(J^{[k_0]},I_1^{[k_1]},\dots,I_i^{[k_i-1]},\dots, I_q^{[k_q]};\overline{M_1}),$$

\noindent where $\overline{M}:=M/x_1M.$ In particular, from this
equality, we have
$$e^j(J^{[k_0]},I_1^{[k_1]},\dots,I_i^{[k_i-1]},\dots,
I_q^{[k_q]};\overline{M_1})\neq 0.$$ \noindent But since
$k_1+\dots +(k_i-1)+\dots +k_q=t-1$,  it follows by inductive
assumption that we can choose
 $(t-1)$ elements $x_2,...,x_t,$ consisting of $k_1$ elements of $I_1$,...,$(k_i-1)$ elements
of $I_i$,..., and $k_q$ elements of $I_q,$  which form an
 (FC)-sequence with respect to
$(J,I_1,\dots,I_q,\overline{M})$. Since $x_1$ is an (FC)-element
with respect $(J,I_1,\dots,I_q;M)$, it follows that
$x_1,x_2,...,x_t$ is an (FC)-sequence with respect to
$(J,I_1,\dots,I_q;M).$

We now prove the sufficiency. Suppose that there exists an
(FC)-sequence $x_1,x_2,...,x_t,$ with respect to
$(J,I_1,\dots,I_q;M),$ of  $t=k_1+...+k_q$ elements consisting of
$k_1$ elements of $I_1$,...,$k_q$ elements of $I_q$. Then it
follows that
$$\mbox{dim}(\mbox{Supp}(M/((x_1,...,x_t)M:{\mathcal{I}}^{\infty}))=j+k_0>0.$$
\noindent Hence, it follows that $\mathcal{I}$ is not contained in
$\sqrt{\mbox{Ann}\overline{M_t}}$. Hence, by Lemma \ref{lema
crucial}, $e^j(J^{[k_0]},I_1^{[0]},\dots,
I_q^{[0]};{\overline{M_t}})=e^j(J;{\overline{M_t}}^*))\neq 0$. But
by $(i)$,
$$e^j(J^{[k_0]},I_1^{[k_1]},\dots, I_q^{[k_q]};M)=e^j(J;{\overline{M_t}}^*)),$$
Hence, the proof of the theorem is complete.
\end{proof}

In the case that ${\mbox ht}({\mathcal
I}+\mbox{Ann}M/{\mbox{Ann}M})>0$, we have that
$$D:=\dim(\mbox{\mbox{Supp}}(M^*))=\dim(\mbox{\mbox{Supp}}(M)).$$
Hence, it follows by Proposition \ref{prop1} and equation
(\ref{equation6}) that there exist positive integers $n_0$ and
${\bold u}$ such that the function
$$\ell\left(\frac{{\bold I}^{\bold r}M_{n+p}} {J^{n}{\bold I}^{\bold r}M_{p}}\right)$$ \noindent is, for $n>n_0,\bold r
>\bold u$ a polynomial of total degree $D=\dim(\mbox{\mbox{Supp}}(M)).$
We get the following theorem.

\begin{thm}\label{Teo4}
In the setup $(1),$ assume ${\mbox ht}({\mathcal
I}+\mbox{Ann}M/{\mbox{Ann}M})>0$. Let $k_0,j,k_1,...,k_q$ be
non-negative integers with sum equal to $D,$ with $k_0>0.$ Then,
if $t=k_1+...+k_q<{\mbox ht}({\mathcal
I}+\mbox{Ann}M/{\mbox{Ann}M})$ we have that
$$e^j(J^{[k_0]},I_1^{[k_1]},\dots, I_q^{[k_q]}; M)=e^j_{BR}(J;\overline{M_t}),$$
\noindent for any (FC)-sequence $x_1,...,x_t$   with respect to
$(J,I_1,\ldots,I_q; M)$ consisting of $k_1$ elements of
$I_1$,...,$k_q$ elements of $I_q,$ where
$\overline{M_t}=M/(x_1,...,x_t)M.$
\end{thm}

\begin{proof}
Set
$\overline{M_t}^*=\overline{M_t}/(0_{\overline{M_t}}:{\mathcal{I}}^{\infty}).$
Then by (\ref{equation7})
$$\overline{M_t}^*=M/((x_1,...,x_t)M:{\mathcal{I}}^{\infty}).$$ By
Theorem \ref{Teo1} (i),
$$e^j(J^{[k_0]},I_1^{[k_1]},\dots, I_q^{[k_q]}; M)=e^j_{BR}(J;\overline{M_t}^*).$$
Since clearly ${\mbox ht}({\mathcal
I}+\mbox{Ann}\overline{M_t}/{\mbox{Ann}\overline{M_t}})>0,$
applying Lemma \ref{lema crucial}, we have
\begin{equation}\label{CRU}
e^j_{BR}(J;\overline{M_t}^*)=e^j_{BR}(J;\overline{M_t}).
\end{equation}

\noindent Thus  $$e^j(J^{[k_0]},I_1^{[k_1]},\dots, I_q^{[k_q]};
M)=e^j_{BR}(J;\overline{M_t})$$ and the proof is complete.
\end{proof}

\subsection{A formula for mixed multiplicities}

We will derive from Theorem \ref{Teo4} a formula for
$e^j(J^{[k_0]},I_1^{[k_1]},\dots, I_q^{[k_q]}; M)$, where
$k_1+...+k_q<{\mbox ht}({\mathcal I}+\mbox{Ann}M/{\mbox{Ann}M})$.
This formula generalizes the main
 results on mixed multiplicities of finite
colength modules given by Kirby and Rees in \cite{Kirby-Rees1} and
the authors in \cite{Bedregal-Perez}.

\begin{prop}\label{Generalized} Let $k_1,\dots, k_q$ be  integers such that $t:=k_1+\cdots + k_q<{\mbox ht}({\mathcal
I}+\mbox{Ann}M/{\mbox{Ann}M})$. Let $y_1,\dots, y_{D-t},x_1,\dots,
x_{t}$ be an (FC)-sequence with respect to $(J,I_1,\ldots,I_q;
M)$, consisting of $D-t$ elements of $J$, $k_1$ elements of $I_1$,
\dots, $k_q$ elements of $I_q.$ Then
$$e^0(J^{[D-t]},I_1^{[k_1]},\dots, I_q^{[k_q]};M)=e_{BR}(y_1,\dots,y_{D-t},x_1,\dots, x_{t};M).$$
\end{prop}
\begin{proof} Since ${\mbox ht}({\mathcal
I}+\mbox{Ann}M/{\mbox{Ann}M})>0$, we have that
$D=\dim(\mbox{\mbox{Supp}}(M^*))=\dim(\mbox{\mbox{Supp}}(M)).$
From Theorem \ref{Teo4} we have that
$$e^0(J^{[D-t]},I_1^{[k_1]},\dots, I_q^{[k_q]};
M)=e_{BR}(J;\overline{M_t}),$$ where
$\overline{M_t}:=M/(x_1,\dots, x_{t})M.$

It is also clear by Proposition \ref{FC JoiRe},   item $(ii),$
that we have that $y_1,\dots,y_{D-t}$ generate a reduction of $J$
module $\overline{M_t}.$ Therefore,
$$e_{BR}(J;{M}_{t})=e_{BR}((y_1,\dots,y_{D-t});{M}_{t}).$$
Without restriction we may assume that $y_{j}$ is not contained in
any minimal prime ideal of $\mbox{Ann}(\overline{M_t}/(y_1,\dots,
y_{j-1}){ M}_{t}),\;j=1,\dots, D-t$. Then for all large $n,$
$$
\begin{array}{cc}
e_{BR}((y_1,\dots,y_{D-t});M_{t})=\ell\left( M_n/(x_1,\dots, x_{t},y_1,\dots,y_{D-t})M_{n-1}\right)-\\
\ell\left((x_1,\dots,
x_{t},y_1,\dots,y_{D-t-1})M_n:_{M_n}y_{D-t}/(x_1,\dots,
x_{t},y_1,\dots,y_{D-t-1})M_{n-1}\right)
\end{array}
$$
by \cite[Proposition 2.6]{Bedregal-Perez}. On the other hand, we
may also assume that for any $j=1,\dots, t,$ $x_j$ do not belong
to $\mathfrak p$, for any associated prime ideal $\mathfrak p$ of
$(x_1,\dots, x_{j-1})$ which do not contain ${\mathcal
I}+\mbox{Ann}M/{\mbox{Ann}M}$. If $\mathfrak p$ is an associated
prime ideal of $(x_1,\dots,x_{j-1})$ with $\dim(M/{\mathfrak
p}M)\geq D-j$, then $\mathfrak p$ do not contain ${\mathcal I}$
because $D-j> D-{\mbox ht}({\mathcal
I}+\mbox{Ann}M/{\mbox{Ann}M})\geq \dim(M/{\mathcal I}M),$ whence
$x_j$ do not belong to $\mathfrak p$. So we can apply
\cite[Proposition 2.6]{Bedregal-Perez} and obtain

$$
\begin{array}{cc}
e_{BR}((x_1,\dots, x_{t},y_1,\dots,y_{D-t}),M)=\ell\left( M_n/(x_1,\dots, x_{t},y_1,\dots,y_{D-t})M_{n-1}\right)-\\
\ell\left((x_1,\dots,
x_{t},y_1,\dots,y_{D-t-1})M_n:_{M_n}y_{D-t}/(x_1,\dots,
x_{t},y_1,\dots,y_{D-t-1})M_{n-1}\right).
\end{array}
$$
This implies  $$e_{BR}((y_1,\dots,y_{D-t});{
M}_{t})=e_{BR}((x_1,\dots, x_{t},y_1,\dots,y_{D-t}),M).$$ Summing
up all equations we get the conclusion.

\end{proof}

\begin{cor}\label{Cor-Generalized} Let $k_1,\dots, k_q$ be  integers such that $t:=k_1+\cdots + k_q<{\mbox ht}({\mathcal
I}+\mbox{Ann}M/{\mbox{Ann}M})$. Let $y_1,\dots, y_{D-t},x_1,\dots,
x_{t}$ be an (FC)-sequence with respect to $(G_1,J,I_1,\ldots,I_q;
M)$, consisting of $j$ elements of $G_1,$ $D-t-j$ elements of $J$,
$k_1$ elements of $I_1$, \dots, $k_q$ elements of $I_q.$ Then
$$e^j(J^{[D-t-j]},I_1^{[k_1]},\dots, I_q^{[k_q]};M)=e_{BR}(y_1,\dots,y_{D-t},x_1,\dots, x_{t};M).$$
\end{cor}

\begin{proof}
We know that for $r_1,\ldots,r_q,n,p\gg 0$
$$\ell \left (\frac{{\bf I}^{\bf r}M_{n+p}}{J^n{\bf I}^{\bf r}M_p}\right )=
\sum_{j+k_0+|{\bf k}|=D}\frac{e^j(J^{[k_0]},I_1^{[k_1]}, \ldots,
I_k^{[k_q]};M)}{j!k_0!{\bf k}!}\;{\bf r}^{\bf
k}n^{k_0}p^j+\cdots$$ \noindent and, since $G_1^uM_v=M_{u+v},$ we
have
$$\ell \left (\frac{{\bf I}^{\bf r}M_{n+u+v}}{J^n{\bf I}^{\bf r}M_{u+v}}\right )=
\ell \left (\frac{{\bf I}^{\bf r}G_1^uM_{n+v}}{J^n{\bf I}^{\bf
r}G_1^uM_{v}}\right )$$ and hence
$$\ell \left (\frac{{\bf I}^{\bf r}M_{n+u+v}}{J^n{\bf I}^{\bf r}M_{u+v}}\right )=
\sum_{s+t+k_0+|{\bf k}|=D}\frac{e^s(G_1^t,J^{[k_0]},I_1^{[k_1]},
\ldots, I_k^{[k_q]};M)}{s!t!k_0!{\bf k}!}\;{\bf r}^{\bf
k}n^{k_0}u^tv^s+\cdots.$$

Now, by making $p=u+v,$ it follows that
$$e^j(J^{[k_0]},I_1^{[k_1]}, \ldots,
I_k^{[k_q]};M)=e^s(G_1^t,J^{[k_0]},I_1^{[k_1]}, \ldots,
I_k^{[k_q]};M),$$ for all non-negative integers $s,t$ such that
$s+t=j.$ In particular we have that
$$e^j(J^{[D-t-j]},I_1^{[k_1]},\dots, I_q^{[k_q]};M)=e^0(G_1^{[j]},J^{[D-t-j]},I_1^{[k_1]},\dots,
I_q^{[k_q]};M).$$ Hence the result follows by applying the
Proposition \ref{Generalized} to the right-hand side of the above
equality.
\end{proof}

\section{Superficial and (FC)-sequences}

Using different kind of sequences one can translate mixed
multiplicities of ideals into Hilbert-Samuel multiplicities. For
instance: In the case of ${\mathfrak m}$-primary ideals,
Risler-Teissier in \cite{Teissier} used superficial sequences and
Rees in \cite{Rees} used joint reductions; in the case of
arbitrary ideals, Vi\^et \cite{Viet8} used (FC)-sequences and
Trung-Verma in \cite{Trung-Verma1} used
$({\varepsilon}_1,\dots,{\varepsilon}_m)$-superficial sequences.
In \cite{Twoviet}, Dinh and Viet proved that every
$({\varepsilon}_1,\dots,{\varepsilon}_m)$-superficial sequences is
in fact an (FC)-sequence and then, using this fact, they proved
that the main result of \cite{Trung-Verma1} could be obtained as
an immediate consequence of \cite[Theorem 3.4]{Viet8}.

In the case of finite colength submodules of $G_1$ one can compute
mixed multiplicities through superficial sequences as in
\cite{Bedregal-Perez} or joint reductions, as in
\cite{Kirby-Rees1} or \cite{Bedregal-Perez}. Here, as a
consequence of Theorem
 \ref{Teo1}, we extend to arbitrary submodules of $G_1$ the
result of Trung-Verma in \cite{Trung-Verma1}. That is, we extend
to this context the main result of \cite{Twoviet}.

\begin{defn}
Set $T=\bigoplus_{{\bf r},p}\frac{{\bf I}^{\bf r}M_{p+q}}{{\bf
I}^{{\bf r}+{\bf 1}}M_p}$. Let $\varepsilon$ be an index with
$1\leq  \varepsilon \leq q$. An element $x\in G$ is an
$\varepsilon$-superficial element for $I_1,\dots, I_q$  with
respect to $M$ if $x\in I_{\varepsilon}$ and the image $x^*$ of
$x$ in $\frac{I_{\varepsilon}G_{p+q}}{I_1\cdots I_{\varepsilon
-1}I_{\varepsilon}^2I_{\varepsilon+1}\cdots I_qG_p}$  is a
filter-regular element in T, i.e., $(0 :_T x^*)_{|{\bf r}|+p} = 0$
for ${\bf r}\gg 0$ and all $p$. Let
${\varepsilon}_1,\dots,{\varepsilon}_m$ be a non-decreasing
sequence of indices with $1\leq {\varepsilon}_i\leq q$. A sequence
$x_1,\dots, x_m$ is an $({\varepsilon}_1,\dots,{\varepsilon}_m)$-
superficial sequence for $I_1,\dots, I_q$ with respect to $M$  if
for $i = 1, \dots ,m$, ${\overline x}_i$ is an
${\varepsilon}_i$-superficial element for ${\overline I}_1,\dots,
{\overline I}_q$ with respect to $\overline{M}$, where ${\overline
x}_i$, ${\overline I}_1,\dots, {\overline I}_q$ are the images of
$x_i$, $I_1,\dots, I_q$ in $G/(x_1\dots,x_{i-1})G$ and
$\overline{M}=M/(x_1\dots,x_{i-1})M$.
\end{defn}

The relation between
$({\varepsilon}_1,\dots,{\varepsilon}_m)$-superficial sequences
and weak-(FC)- sequences is given  by the following Proposition,
which extend to this context the main result of \cite{Twoviet}.

\begin{prop}\label{thm3}. Let $I_1,\dots,I_q$ be $R$-submodules of $G_1$. Let $x\in G$  be
an $\varepsilon$-superficial element for $I_1,\dots,I_q$. Then $x$ is a weak-(FC)-element with respect
to $(I_1,\dots,I_q;M)$.
\end{prop}

\begin{proof}
Let $x$ be an $\varepsilon$-superficial element for $I_1,\dots,
I_q$ with respect to $M.$ Without loss of generality, we may
assume that $\varepsilon=1.$ Then

\begin{equation}\label{equation11}
({\bf I}^{{\bf r+1}+\delta(1)}M_p:x)\cap {\bf I}^{\bf
r}M_{p+q}={\bf I}^{\bf r+1}M_{p}
\end{equation}
\noindent for ${\bf r}\gg 0$ and for all $p.$ This equality
implies
\begin{equation}\label{equation12}
({\bf I}^{{\bf r+1}+\delta(1)}M_p:x)\cap {\bf I}^{{\bf
r+1}-\delta(1)}M_{p+1}={\bf I}^{\bf r+1}M_p
\end{equation}
\noindent for ${\bf r}\gg 0$ and for all $p.$ We prove by
induction on $k\geq 2$ that
\begin{equation}\label{equation13}
({\bf I}^{{\bf r+1}+(k-1)\delta(1)}M_p:x)\cap {\bf I}^{{\bf
r+1}-\delta(1)}M_{p+k-1}={\bf I}^{{\bf r+1}+(k-2)\delta(1)}M_{p}
\end{equation}
\noindent for ${\bf r}\gg 0$  and for all $p.$ The case $k=2$
follows from equality (\ref{equation12}). Assume now that
 $$({\bf I}^{{\bf r+1}+(k-1)\delta(1)}M_p:x)\cap {\bf I}^{{\bf r+1}-\delta(1)}M_{p+k-1}={\bf I}^{{\bf r+1}+(k-2)\delta(1)}M_{p}$$
 \noindent for ${\bf r}\gg 0$  and for all $p.$ Then
 $$
 \begin{array}{lll}
 \vspace{0.3cm}
 ({\bf I}^{{\bf r+1}+k\delta(1)}M_p:x)\cap {\bf I}^{{\bf r+1}-\delta(1)}M_{p+k}\!\!&\!\!=\!\!&({\bf I}^{{\bf r+1}+k\delta(1)}M_p:x)\cap\\
 \vspace{0.3cm}
 \!\!\!\!&&\!\!\!\!({\bf I}^{{\bf r+1}+(k-1)\delta(1)}M_{p+1}:x)\cap{\bf I}^{{\bf r+1}-\delta(1)}M_{p+k}\\
 \vspace{0.3cm}
 \!\!&\!\!=\!\!&({\bf I}^{{\bf r+1}+k\delta(1)}M_p:x)\cap {\bf I}^{{\bf r+1}+(k-2)\delta(1)}M_{p+1}\\
 \vspace{0.3cm}
 \!\!&\!\!=\!\!&{\bf I}^{{\bf r+1}+(k-1)\delta(1)}M_p.
 \end{array}
 $$
 \noindent for ${\bf r}\gg 0$  and for all $p.$ The last equality is gotten from equality (\ref{equation12}).
 Hence the induction is complete and we get equality (\ref{equation13}). Denote by ${\bf{\cal I}}^{\bf r}$
 the ideal in $G$ generated by ${\bf I}^{\bf r}.$ It follows that
 for ${\bf r}\gg 0.$

 $$
 \begin{array}{lll}
 \vspace{0.3cm}
 (0:x)\cap {\bf{\cal I}}^{{\bf r+1}-\delta(1)}M&=&(\bigcap_{k\geq 2}{\bf{\cal I}}^{{\bf r+1}+(k-1)\delta(1)}M:x)\cap {\bf{\cal I}}^{{\bf r+1}-\delta(1)}M\\
 \vspace{0.3cm}
  &=&\left(\bigcap_{k\geq 2}({\bf{\cal I}}^{{\bf r+1}+(k-1)\delta(1)}M:x)\right)\cap {\bf{\cal I}}^{{\bf r+1}-\delta(1)}M\\
  \vspace{0.3cm}
 &=&\bigcap_{k\geq 2}\left(({\bf{\cal I}}^{{\bf r+1}+(k-1)\delta(1)}M:x)\cap {\bf{\cal I}}^{{\bf r+1}-\delta(1)}M\right)\\
 \vspace{0.3cm}
 &=&\bigcap_{k\geq 2}{\bf{\cal I}}^{{\bf r+1}+(k-2)\delta(1)}M\\
 \vspace{0.3cm}
 &=&0,
 \end{array}
 $$
 \noindent that is, $(0_M:x)\cap {\bf{\cal J}}^{\bf n}M=0$ for ${\bf n}\gg 0.$ Here ${\bf{\cal J}}$ is the ideal in $G$ generated by $I_1\cdots I_q$.
 Hence $(0_M:x)\subseteq(0_M:{\bf{\cal J}}^{\infty}).$ Thus $x$  satisfies the condition $(FC_2)$.
 Now we need to prove that

 \begin{equation}\label{equation16}{\bf I}^{\bf r}M_p\cap xM_{|{\bf r}-\delta(1)|+p}=x{\bf I}^{{\bf
 r}-\delta(1)}M_p,\end{equation}
 \noindent for ${\bf r}\gg 0$ and for all $p.$

 From equality (\ref{equation11}) we get
 $${\bf{\cal I}}^{{\bf r+1}+\delta(1)}M\cap x {\bf{\cal I}}^{\bf v}M=x{\bf{\cal I}}^{\bf r+1}M,$$
 \noindent for ${\bf r}\gg 0$, where $v_i=r_i$ or $v_i=r_i+1$, $i=1,\dots, q.$  Using this formula we can easily show that

 $${\bf{\cal I}}^{{\bf r}}M\cap x {\bf{\cal I}}^{\bf v}M=x{\bf{\cal I}}^{\bf t}M,$$
 \noindent  for ${\bf r}\gg 0$, ${\bf t}\gg 0,$ where $t_i=\mbox{max}\{r_i,v_i\}$,$i=1,\dots, q$.

 By Artin-Rees Lemma, there exists $(c_1,\dots,c_q)$ with $c_1>0$ such that

 $${\bf{\cal I}}^{{\bf r}}M\cap xM\subset x{\bf{\cal I}}^{\bf r-c}M,$$
 \noindent for all ${\bf r}\geq {\bf c}$. Therefore

 $$
 \begin{array}{lll}
 {\bf{\cal I}}^{{\bf r}}M\cap xM&=&{\bf{\cal I}}^{{\bf r}}M\cap x{\bf{\cal I}}^{\bf r-c}M\\
 &=&x{\bf{\cal I}}^{{\bf r}-\delta(1)}M,
 \end{array}
 $$
 \noindent for ${\bf r}\gg 0.$ Equality (\ref{equation16}) now
 follows by concentrating in degree $|{\bf r}|+p$ in the last
 equality.
\end{proof}

The following theorem generalizes the main result of Trung-Verma
in \cite[Theorem 1.5]{Trung-Verma1} for modules.

\begin{thm}
Set $D = \dim\left({\mbox{Supp}}({M^*})\right)$. Let $k_0,
k_1,\dots, k_q$ be nonnegative integers such that $k_0 +
k_1+\cdots +k_q = D-1$. Assume that
${\varepsilon}_1,\dots,{\varepsilon}_m$ $(m = k_1+\cdots +k_q)$ is
a non-decreasing sequence of indices consisting of $k_1$ numbers
$1, \dots ,k_q$ numbers $q$. Let $Q$ be any ideal generated by an
$({\varepsilon}_1,\dots,{\varepsilon}_m)$-superficial sequence for
$J$, $I_1,\dots, I_q$. Then $e^0(J^{[k_0+1]}, I^{[k_1]}_1,\dots,
I^{[k_q]}_q;M) \neq 0$  if and only if
$\dim(\mbox{Supp}(M/(QM:{\mathcal I}^{\infty})))=k_0+1$. In this
case,
$$e^0(J^{[k_0+1]}, I^{[k_1]}_1,\dots,I^{[k_q]}_q;M)=e_{BR}(J;M/(QM:{\mathcal I}^{\infty})).$$
\end{thm}
\begin{proof} Assume that $Q=(x_1,\dots,x_m)$, where $x_1,\dots,x_m$ is an $({\varepsilon}_1,\dots,{\varepsilon}_m)$-superficial sequence for
$J$, $I_1,\dots, I_q$ with respect to $M.$ Thus, by Proposition
\ref{thm3}, $x_1,\dots,x_m$ is a weak-(FC)-sequence with respect
to $(J,I_1,\dots, I_q;M)$. Hence
$$\dim(\mbox{Supp}(M/(QM:{\mathcal I}^{\infty})))\leq D-m=k_0+1,$$
\noindent with equality if and only if $x_1,\dots,x_m$ is an
(FC)-sequence by Proposition \ref{prop3}. Therefore from Theorem
\ref{Teo1} we have the result.
\end{proof}

\section{Applications}

Let $(R, \mathfrak{m})$ be a Noetherian local ring and $E$ a
submodule of the free $R$-module $R^p.$ The symmetric algebra
$G:=\mbox{Sym}(R^p)=\oplus S_n(R^p)$ of $R^p$ is a polynomial ring
$R[T_1,\ldots ,T_p].$ If $h=(h_1,\ldots ,h_p)\in R^p,$ then we
define the element $w(h)=h_1T_1 + \ldots +h_pT_p\in
S_1(R^p)=:G_1.$ We denote by ${\cal R}(E):=\oplus {\cal R}_n(E)$
the subalgebra of $G$ generated in degree one by $\{w(h): h\in
E\}$ and call it the {\it  Rees algebra} of $E$.  Given any
finitely generated $R$-module $N$ consider the graded $G$-module
$M:=G\otimes_R N.$

We are now ready to introduce the main object of this paper, the
mixed multiplicities for a family of $R$-submodules of $R^p.$ Here
the linear submodules of $G_1$ of the previous sections will be
replaced by a module $E.$ Let $E_1,\ldots,E_q$ be $R$-submodules
of $R^p$ and denote by $I_i$ the $R$-submodule of $G_1$ given by
${\cal R}_1(E_i),$ for all $i=1,\dots, q.$ Let ${\mathcal I}$ be
the ideal of $G$ generated by $I_1\cdots I_q$. We translate into
this context the basic definitions and results of the previous
sections.

A sequence of elements $h_1,\ldots, h_q,$ with $h_i\in E_i,$ is an
{\it (FC)-sequence}  with respect to $(E_1,\dots,E_q; N)$ if for
all $i=1,\ldots,q,$ $h_i\in E_i$ are such that the sequence
$w(h_1),\ldots, w(h_k)$ is an (FC)-sequence in $G$ with respect to
$(I_1,...,I_q; M).$

Let $F$ be a finitely generated $R$-submodule of $R^p$ of finite
colength. Define the  function
$$h(n,p,\bold r):=\ell\left(\frac{{\cal R}_{r_1}(E_1)\cdots
{\cal R}_{r_q}(E_q)M_{n+p}} {{\cal R}_{n}(F){\cal
R}_{r_1}(E_1)\cdots {\cal R}_{r_q}(E_q)M_p}\right).$$ \noindent If
$n,p,\bold r\gg 0$ then $h(n,p,\bold r)$ becomes a polynomial of
total degree $D:=\dim((N/0_M:_N{\mathcal I}^{\infty}))+p-1$. If we
write the terms of total degree $\dim((N/0_M:_N{\mathcal
I}^{\infty}))+p-1$ of this polynomial in the form
$$B(n,p,{\bold r})=
\sum_{j+k_0+|{\bf k}|=D}\;\frac{1}{j!k_0!{\bf
k}!}e^j(F^{[k_0]},E_1^{[k_1]},\dots, E_q^{[k_q]};N)n^{k_0}p^j{\bf
r}^{\bf k}.$$ \noindent The coefficients
$e^j(F^{[k_0]},E_1^{[k_1]},\dots, E_q^{[k_q]};N)$ are called the
$j^{th}$-{\it mixed multiplicity of} $(F,E_1,\dots, E_q)$ with
respect to $N$ of the type $(k_0,k_1,\dots k_q).$

Theorem \ref{Teo1}  immediately gives the following result.

\begin{thm}\label{mod1}
 Let $F,E_1,...,E_q$ be as in the beginning of this section. Set
 $J={\cal R}_1(F).$
Let $k_0,j,k_1,...,k_q$ be non-negative integers with sum equal to
$D.$ Assume that $D>0.$ Then
\begin{itemize}
\item [(i)] $$e^j(F^{[k_0]},E_1^{[k_1]},\dots,E_q^{[k_q]};N)=e^j({J};\overline{M}^*),$$
\noindent for any (FC)-sequence $x_1,...,x_t$,  with respect to
$(F,E_1,...,E_q; M),$  of $t=k_1+...+k_q$ elements of $R^p$
consisting of $k_1$ elements of $E_1$,...,$k_q$ elements of $E_q,$
where
$\overline{M}^*=M/((w(x_1),...,w(x_t))M:{\mathcal{I}}^{\infty}).$

\item [(ii)] If $k_0>0,$ then $e^j(F^{[k_0]},E_1^{[k_1]},\dots,E_q^{[k_q]};N)\neq 0,$
if and only if there exist an (FC)-sequence with respect to
$(F,E_1,...,E_q; M),$ of $k_1+...+k_q$ elements   consisting of
$k_1$ elements of $E_1$,...,$k_q$ elements of $E_q.$
\end{itemize}
\end{thm}

Theorem \ref{Teo4}  immediately gives the following result.
\begin{thm}\label{mod2}
Let $F,E_1,\ldots,E_q$ be as in the beginning of this section. Set
 $J={\cal R}_1(F).$
Assume ${\mbox ht}({\mathcal I}+\mbox{Ann}M/{\mbox{Ann}M})>0$. Let
$j,k_0,k_1,...,k_q$ be non-negative integers with sum equal to
$D,$ with $k_0>0.$ Then, if $t=k_1+...+k_q<{\mbox ht}({\mathcal
I}+\mbox{Ann}M/{\mbox{Ann}M})$ we have that
$$e^j(F^{[k_0]},E_1^{[k_1]},\dots, E_q^{[k_q]}; N)=e^j(J;\overline{M}),$$
\noindent for any (FC)-sequence $x_1,...,x_t$ with respect to
$(F,E_1,\ldots,E_q; M)$ consisting of $k_1$ elements of
$E_1$,...,$k_q$ elements of $E_q$ where
$\overline{M}=M/(w(x_1),...,w(x_t))M.$
\end{thm}

Corollary \ref{Cor-Generalized} immediately gives the following
result.

\begin{prop} Let $t=k_1+...+k_q<{\mbox ht}({\mathcal
I}+\mbox{Ann}M/{\mbox{Ann}M})$. Let $y_1,\dots,y_{D-t},x_1,\dots,
x_{t}$ be an (FC)-sequence with respect to $(R^p,F,E_1,\ldots,
E_q; N),$ consisting of
 $j$ elements of $R^p$, $D-t-j$ elements of $F,$ $k_1$ elements of $E_1$,...,$k_q$ elements
of $E_q.$ Then,
$$e^j(F^{[D-t-j]},E_1^{[k_1]},\dots, E_q^{[k_q]};
N)=e_{BR}(x_1,\dots, x_{t},y_1,\dots,y_{D-t}; N).$$
\end{prop}

From these facts, we have a  similar result to that of Kirby and
Rees in \cite{Kirby-Rees1} and the authors in
\cite{Bedregal-Perez}, but in terms of (FC)-sequences.
\begin{thm}\label{mod3}
Suppose   that  $E_1,\ldots,E_q$ are $R$-submodules of $R^p$ of
finite colength. Let $j,k_1,\ldots,k_q\in \N$ with $j+|{\bf
k}|=d+p-1,$ where $\dim N=\dim R=d.$ Let $x_1,\ldots,x_{d+p-1}$ be
a a weak-(FC)-sequence with respect to $(R^p,E_1,\dots, E_q;N)$
consisting of
 $j$ elements of $R^p$, $k_1$ elements of $E_1$,...,$k_q$ elements
of $E_q.$ Then,
$$e^j(E_1^{[k_1]}, \ldots, E_q^{[k_q]},N)=e_{BR}((x_1,\ldots,x_{d+p-1}),N).$$
\end{thm}

\end{document}